\documentclass[reqno]{amsart}
\usepackage{amssymb,enumerate,mathrsfs}
\usepackage{mathrsfs}
\usepackage[pdftex=true,
	bookmarks=true,
	bookmarksopen=true,
	bookmarksnumbered=true,
	pdftoolbar=true,
	pdfmenubar=true,
	pdffitwindow=false,
	pdfstartview={FitH},
	colorlinks=true,
	linkcolor=red,
	citecolor=blue,
	breaklinks=true]{hyperref}

\numberwithin{equation}{section}

\newtheorem{theorem}{Theorem}[section]
\newtheorem{proposition}[theorem]{Proposition}
\newtheorem{lemma}[theorem]{Lemma}

\theoremstyle{definition}

\newcommand{\R}{\mathbb R}

\newcommand{\dd}{\mathrm{d}}

\newcommand{\eps}{\epsilon}
\newcommand{\im}{\operatorname{im}}
\newcommand{\ind}{\operatorname{ind}}
\newcommand{\Tr}{\operatorname{Tr}}
\newcommand{\tr}{\operatorname{tr}}
\newcommand{\End}{\operatorname{End}}

\newcommand{\dx}{\dd x}

\newcommand{\VF}{\mathfrak X}
\newcommand{\Tensor}{\mathfrak T}
\newcommand{\nor}{\operatorname{nor}}

\newcommand{\II}{I\!I}

\begin{document}
\title[Geometry of embedded pseudo-Riemannian surfaces]{The geometry of embedded pseudo-Riemannian surfaces in terms of Poisson brackets}

\author{Peter Hintz}
\address{Mathematisches Institut, Bunsenstra\ss e 3-5, 37073 G\"ottingen, Germany}
\email{peter.hintz@stud.uni-goettingen.de}


\begin{abstract}
  Arnlind, Hoppe and Huisken showed in \cite{ArnHopHuis} how to express the Gauss and mean curvature of a surface embedded in a Riemannian manifold in terms of Poisson brackets of the embedding coordinates. We generalize these expressions to the pseudo-Riemannian setting and derive explicit formulas for the case of surfaces embedded in $\R^m$ with indefinite metric.
\end{abstract}

\maketitle

\section{Introduction}

\noindent
Motivated by certain equations of Membrane Theory, Arnlind, Hoppe and Huisken found a way to express geometric quantities of a surface $\Sigma$ embedded in a Riemannian manifold in a purely algebraic way \cite{ArnHopHuis}: They showed using the canonical Poisson bracket on $C^\infty(\Sigma)$ that the Gauss curvature of a surface embedded in the euclidean $\R^m$ via the coordinates $x^1,\ldots,x^m$ is given by
\begin{equation}
\label{eq:Kintro}K=-\frac{1}{8(m-3)!}\sum_{L\in\{1,\ldots,m\}^{m-3}}\eps_{jklL}\eps_{irnL}\{x^i,\{x^k,x^l\}\}\{x^j,\{x^r,x^n\}\},
\end{equation}
thereby also providing a starting point to generalize the notion of curvature to certain non-commutative spaces, e.~g. matrix models that relate the Poisson bracket to the commutator of matrices.

The crucial point in the derivation of \ref{eq:Kintro} turns out to be the explicit construction of normal vectors to $\Sigma$ in terms of Poisson brackets of the embedding coordinates. In section \ref{sec:construction}, we will perform this construction for surfaces $\Sigma$ embedded in a pseudo-Riemannian manifold. In theorem \ref{theorem:GaussMeanRnnuFull}, we will then derive a generalization of equation \ref{eq:Kintro} to the pseudo-Riemannian setting.

\section{Preliminaries}
\label{sec:prelim}

\noindent
We use the following notation (cf. \cite{ArnHopHuis}, \cite{ONeill}): Let $M$ be an $m$-dimensional pseudo-Riemannian manifold and $\Sigma\subset M$ a 2-dimensional orientable pseudo-Riemannian submanifold with codimension $p=m-2$. The metric tensors of $M$ and $\Sigma$ are denoted by $\bar g$ and $g$, the Levi-Civita connections by $\bar\nabla$ and $\nabla$, the curvature tensors by $\bar R$ and $R$ (with the convention $R_{XY}Z=\nabla_{[X,Y]}Z-[\nabla_X,\nabla_Y]Z$). Indices $a,b,c,d,p,q$ run from $1$ to $2$, indices $i,j,k,l,r,s,t,v,w$ from $1$ to $m$ and $A,B$ from $1$ to $p$. $x^i$ and $u^a$ are local coordinates of $M$ and $\Sigma$, the Christoffel symbols are $\bar\Gamma_{ij}^k$ and $\Gamma_{ab}^c$, respectively. The metric of $M$ in the coordinates $x^i$ is given by the matrix $(\bar g_{ij})$, the metric of $\Sigma$ in the coordinates $u^a$ is given by $(g_{ab})=\bar g(e_a,e_b)$, $e_a=(\partial_a x^i)\partial_i$ being an oriented basis of the tangent space $T\Sigma$. The inverses of these matrices are denoted by $(\bar g^{ij})$ and $(g^{ab})$. We write $g=\det(g_{ab})$. Furthermore, let $N_A=N_A^m\partial_m$ be a smooth local orthonormal frame of the normal bundle $T\Sigma^\perp\subset TM$, that is, $\bar g(N_A,N_B)=\delta_{AB}\sigma_A$ with $\sigma_A=\pm 1$. Finally, let $\eps^{ab}$ be the totally antisymmetric tensor of rank 2.

$\VF(M)$ denotes the space of vector fields on $M$, $\Omega^1(M)$ the space of 1-forms, $\Tensor^r_s(M)$ the space of tensor fields on $M$ of type $(r,s)$. The space of smooth maps $Z\colon\Sigma\to TM$ with $\pi_{TM\to M}\circ Z=id_\Sigma$ is denoted by $\bar\VF(\Sigma)$. $\bar\Omega^1(\Sigma)$ and $\bar\Tensor^r_s(\Sigma)$ are defined analogeously. For $p\in\Sigma$ we have the orthogonal projections $\tan_p\colon T_pM\to T_p\Sigma$ and $\nor_p=1-\tan_p\colon T_pM\to T_p\Sigma^\perp$.

  For $X,Y\in\VF(\Sigma)$, the Gauss and Weingarten equations
  \begin{align}
\label{eq:Gauss2}\bar\nabla_X Y&=\nabla_X Y+\II(X,Y)\\
\label{eq:Weingarten}\bar\nabla_X N_A&=-W_A(X)+D_XN_A
  \end{align}
  hold with $W_A(X)=-\tan\bar\nabla_X N_A$ and $D_XN_a=\nor\bar\nabla_X N_A$. The orthonormal decomposition of the shape tensor $\II$ is
  \begin{equation}
  \label{eq:ShapeOrthonormalExpansion}
    \II(X,Y)=\sum_{A=1}^p\sigma_A\bar g(\II(X,Y),N_A)N_A=\sum_{A=1}^p\sigma_A h_A(X,Y)N_A
  \end{equation}
  with
  \[
    h_A(X,Y)=\bar g(\II(X,Y),N_A)=-\bar g(X,\bar\nabla_YN_A).
  \]
  Setting $h_{A,ab}=h_A(e_a,e_b)$, we have
  \begin{align}
\label{eq:hAab}h_{A,ab}&=h_{A,ba}\\
\label{eq:WAab}(W_A)^a_b&=g^{ac}W_{A,bc}=g^{ac}g(e_b,W_Ae_c)=g^{ac}h_{A,bc}=g^{ac}h_{A,cb}.
  \end{align}
  Using \eqref{eq:ShapeOrthonormalExpansion}, we can write the Gauss curvature of $\Sigma$ as
  \begin{align}
    K=\,&\frac{1}{g}\left[\bar g(\bar R(e_1,e_2)e_1,e_2)+\bar g(\II(e_1,e_1),\II(e_2,e_2))-\bar g(\II(e_1,e_2),\II(e_1,e_2))\right]\nonumber\\
\label{eq:GaussWithDet}&=\frac{1}{g}\left(\bar g(\bar R(e_1,e_2)e_1,e_2)+\sum_{A=1}^p\sigma_A\det\left(h_{A,ab}\right)\right).
  \end{align}
  The mean curvature vector is given by
  \begin{equation}
  \label{eq:MeanWithTrace}H=\frac{1}{2}g^{ab}\II(e_a,e_b)=\frac{1}{2}\sum_{A=1}^p\sigma_A(W_A)^a_aN_A=\frac{1}{2}\sum_{A=1}^p\sigma_A(\tr W_A)N_A.
  \end{equation}

\subsection{The Poisson algebra $C^\infty(\Sigma)$}

  We have the following standard construction: Let $\omega=\rho\,\dd u_1\wedge\dd u_2$ be a symplectic form on $\Sigma$. There is a unique map $X\colon C^\infty(\Sigma)\to\VF(\Sigma)$, $f\mapsto X_f$, with $\omega(X_f,Y)=\dd f(Y)$ for all $Y\in\VF(\Sigma)$. Then
  \begin{equation}
    \label{eq:PoissonBracket}\{f,g\}:=\omega(X_f,X_g)=\frac{1}{\rho}\eps^{ab}(\partial_a f)(\partial b g)
  \end{equation}
  is a Poisson bracket on $C^\infty(\Sigma)$.

  We continue by recalling some of the results obtained in \cite{ArnHopHuis} which are still valid in the pseudo-Riemannian setting. Define the functions
  \begin{align}
    \mathcal P^{ij}&=\{x^i,x^j\}\\
	\label{eq:SAexplicit}\mathcal S_A^{ij}&=\frac{1}{\rho}\eps^{ab}(\partial_ax^i)(\bar\nabla_bN_A)^j=\{x^i,N_A^j\}+\{x^i,x^k\}\bar\Gamma^j_{kl}N_A^l.
  \end{align}
  They are components of $\bar\Tensor^2_0(\Sigma)$ tensors because by lowering the second index we obtain maps $P,S_A\in\End(\bar\VF(\Sigma))$,
  \begin{align*}
    \mathcal P(X)&=\mathcal P^{ik}\bar g_{kj}X^j\partial_i=-\frac{1}{\rho}\bar g(X,e_a)\eps^{ab}e_b\\
    \mathcal S_A(X)&=\mathcal S_A^{ik}\bar g_{kj}X^j\partial_i=-\frac{1}{\rho}\bar g(X,\bar\nabla_a N_A)\eps^{ab}e_b.
  \end{align*}
  Proposition 3.4 in \cite{ArnHopHuis} states that
  \begin{align}
    \label{eq:P2trace}\tr\mathcal P^2&=-2\frac{g}{\rho^2}\\
    \label{eq:SAtrace}\tr\mathcal S_A^2&=-\frac{2}{\rho^2}\det\left(h_{A,ab}\right).
  \end{align}
  
  Defining $\mathcal B_A\in\End(\bar\VF(\Sigma))$ to be the composition $\mathcal B_A=\mathcal P\mathcal S_A$, one can check (cf. proposition 3.3 in \cite{ArnHopHuis}) that for $X\in\bar\VF(\Sigma)$,
  \begin{equation}
  \label{eq:BAandWeingarten1}
    \mathcal B_A(X)=-\frac{g}{\rho^2}\bar g (X,\bar\nabla_a N_A)g^{ab}e_b.
  \end{equation}
  In particular, for $Y\in\VF(\Sigma)$, $\mathcal B_A(Y)=\frac{g}{\rho^2}W_A(Y)$. Consequently, $\tr\mathcal B_A=\frac{g}{\rho^2}\tr W_A$.
  
  The equations \eqref{eq:GaussWithDet} and \eqref{eq:MeanWithTrace} now yield the formulas
  \begin{align}
    \label{eq:Gauss}K&=\frac{1}{g}\bar g(\bar R(e_1,e_2)e_1,e_2)-\frac{\rho^2}{2g}\sum_{A=1}^p\sigma_A\tr\mathcal S_A^2\\
  	\label{eq:Mean} H&=\frac{\rho^2}{2g}\sum_{A=1}^p\sigma_A(\tr\mathcal B_A)N_A.
  \end{align}
  (Note the occurrence of the signs $\sigma_A$ in the pseudo-Riemannian setting as opposed to the Riemannian setting.) In the case $M=\R^m_\nu$ with metric $\bar g_{ij}=\delta_{ij}\bar g_j$, $\bar g_1=\cdots=\bar g_\nu=-1$, $\bar g_{\nu+1}=\cdots=\bar g_m=1$, these formulas become
  \begin{align}
  \label{eq:GaussRnnu}K&=-\frac{\rho^2}{2g}\sum_{A=1}^p\sigma_A\bar g_i\bar g_j\{x^i,N_A^j\}\{x^j,N_A^i\}\\
  \label{eq:MeanRnnu}H&=\frac{\rho^2}{2g}\sum_{A=1}^p\sigma_A\bar g_i\bar g_j\{x^i,x^j\}\{x^j,N_A^i\}N_A^k\partial_k.
  \end{align}

\section{Construction of normal vectors}
\label{sec:construction}

  The formulas \eqref{eq:GaussRnnu} and \eqref{eq:MeanRnnu} can be written solely in terms of Poisson brackets provided that normal vectors can be expressed by Poisson brackets of the embedding coordinates $x^i$. We now generalize the construction of normal vectors in \cite{ArnHopHuis} to the pseudo-Riemannian case.

  First, we introduce further notation: For multi-indices $I=(i_1,\ldots,i_k)$ and $J=(j_1,\ldots,j_k)$ let $\bar g^{IJ}:=\prod_{l=1}^k \bar g^{i_lj_l}$, $\bar g_{IJ}:=\prod_{l=1}^k \bar g_{i_lj_l}$ and $\dx^I:=\dx^{i_1}\otimes\cdots\otimes\dx^{i_k}\in\Omega^1(M)\otimes\cdots\otimes\Omega^1(M)\cong\Tensor^0_k(M)$. Since on $\Omega^1(M)$ we have the scalar product (to be understood as a scalar product at each point) $\hat g(\eta_i\dx^i,\omega_j\dx^j)=\eta_i\omega_j\bar g^{ij}$, we obtain a scalar product on $\Tensor^0_{p-1}(M)$ (again, pointwise),
  \begin{equation}
    \bar g_\otimes(\eta_I\dx^I,\omega_J\dx^J)=\eta_I\omega_J\bar g^{IJ}.
  \end{equation}

  Now, proposition 4.2 in \cite{ArnHopHuis} states that for any multi-index $J\in\{1,\ldots,m\}^{p-1}$,
	\begin{equation}
	\label{eq:NormalVector}
	  Z_J=\frac{\rho}{2\sqrt{|g|(p-1)!}}\bar g^{ij}\eps_{jklJ}\mathcal P^{kl}\partial_i\in T\Sigma^\perp,
	\end{equation}
  where $\eps_{jklJ}$ denotes the Levi-Civita tensor of $M$.

  By raising the $p-1$ indices, we can also regard the tensor $Z\in\bar\Tensor^1_{p-1}(\Sigma)$ as a tensor of type $((p-1)+1,0)$. Explicitly,
  \begin{equation}
  \label{eq:ZTensorUp}
    Z^J=\bar g^{JK}Z_K=\frac{\rho}{2\sqrt{|g|(p-1)!}}\bar g_{kr}\bar g_{ln}\eps^{irnJ}\{x^k,x^l\}\partial_i.
  \end{equation}

  From now on, we use $\delta:=\ind\bar g-\ind g$, $\ind$ being the index of the respective bilinear form.
  \begin{lemma}
    \label{lemma:Zsum}
    $Z_K^iZ^{Kj}=(-1)^\delta\left(\bar g^{ij}+\frac{\rho^2}{g}(\mathcal P^2)^{ij}\right).$
  \end{lemma}
  \begin{proof}
    Use $\mathcal P^{ij}=-\mathcal P^{ji}$, equation \eqref{eq:P2trace} and $\eps_{rstK}\eps^{jvwK}=(-1)^{\ind\bar g}(p-1)!(\delta^j_r\delta^v_s\delta^w_t+\delta^j_s\delta^v_t\delta^w_r+\delta^j_t\delta^v_r\delta^w_s - \delta^j_s\delta^v_r\delta^w_t-\delta^j_t\delta^v_s\delta^w_r-\delta^j_r\delta^v_t\delta^w_s)$.\qedhere
  \end{proof}

  Define the tensor
  \begin{equation}
  \label{def:Zmap}
    \mathcal Z^{IJ}=(-1)^\delta\bar g(Z^I,Z^J).
  \end{equation}
  By lowering the multi-index $J$, we can regard $\mathcal Z$ as an endomorphism of $\bar\Tensor^0_{p-1}(\Sigma)$,
  \begin{equation}
	\omega_I\dx^I\mapsto\mathcal Z^I_J\omega_I\dx^J.
  \end{equation}

  \begin{proposition}
  \label{prop:ZProperties}
    The map $\mathcal Z\in\End(\bar\Tensor^0_{p-1}(\Sigma))$ satisfies
	\begin{enumerate}
	  \item $\mathcal Z^2=\mathcal Z$.
	  \item $\Tr\mathcal Z=p$.
	  \item For $\eta,\omega\in\Tensor^0_{p-1}$, $\bar g_\otimes(\mathcal Z\eta,\omega)=\bar g_\otimes(\eta,\mathcal Z\omega)$.
	\end{enumerate}
  \end{proposition}
  \begin{proof}
    The first two equations follow from lemma \ref{lemma:Zsum} (cf. lemma 4.3 in \cite{ArnHopHuis}). To prove the third equation, we calculate
	\begin{align*}
	  \bar g_\otimes(\mathcal Z\eta,\omega)&=\bar g_\otimes(\mathcal Z^I_J\eta_I\dx^J,\omega_K\dx^K)=(-1)^\delta\eta_I\omega_K\bar g^{JK}\bar g_{ij}Z_J^iZ^{Ij}\\
	    &=(-1)^\delta\eta_I\omega_K\bar g_{ij}Z^{Ki}Z^{Ij}=(-1)^\delta\eta_I\omega_K\bar g_{ij}Z^{Ii}Z^{Kj}=\cdots\\
		&=\bar g_\otimes(\eta,\mathcal Z\omega).\qedhere
	\end{align*}
  \end{proof}

  \begin{proposition}
  \label{prop:OrthProj}
    Let $(V,g)$ be a scalar product space. Let $P\colon V\to V$ be linear with $P^2=P$ and $g(Pv,w)=g(v,Pw)$ for all $v,w\in V$. Then $W=\im P$ is a non-degenerate subspace of $V$.
  \end{proposition}
  \begin{proof}
    This follows from $v=Pv+(v-Pv)$ for $v\in V$, i.~e. $V=W+W^\perp$.
  \end{proof}

  We can now construct an orthonormal basis of $T\Sigma^\perp$. The main ingredient of the construction is the non-degeneracy of the image of $\mathcal Z$.
  \begin{proposition}
  \label{prop:NormalVectors}
    \begin{enumerate}
	  \item The vector space $\bar\Tensor^0_{p-1}(\Sigma)$ with the metric $\bar g_\otimes$ has an orthonormal basis $E^I=E^I_J\dx^J$ of eigenvectors of $\mathcal Z\in\End(\bar\Tensor^0_{p-1}(\Sigma))$.
	  \item Define the normal vectors $\hat N^I=Z^JE^I_J$. Then exactly $p=\dim T\Sigma^\perp$ of the $\hat N^I$ are different from $0$. The $E^I$ can be chosen in such a way that the non-vanishing $\hat N^I$ are a smooth orthonormal frame of $T\Sigma^\perp$.
	\end{enumerate}
  \end{proposition}
  \begin{proof}
    \begin{enumerate}
	  \item follows immediately from propositions \ref{prop:ZProperties} and \ref{prop:OrthProj}.
	  \item Let $\mu^I$ be the eigenvalue of the eigenvector $E^I$ of $\mathcal Z$. Then
	    \begin{align*}
		  \bar g(\hat N^I,\hat N^J)&=(-1)^\delta E^J_L\bar g^{LM}\mathcal Z^K_M E^I_K=(-1)^\delta E^J_L\bar g^{LM}\mu^I E^I_M\\
		   &=(-1)^\delta\mu^I\bar g_\otimes(E^I,E^J)=\pm\mu^I\delta^{IJ}.
		\end{align*}
		As $\mu^I\in\{0,1\}$ and $\Tr\mathcal Z=p$, exactly $p$ of the $\mu^I$ are $1$. The corresponding $\hat N^I$ span $T\Sigma^\perp$. Since $T\Sigma^\perp$ is non-degenerate, the other $\hat N^J\in T\Sigma^\perp$ vanish.\qedhere
	\end{enumerate}
  \end{proof}

  To obtain explicit formulas for the curvature of $\Sigma$, we need to slightly extend lemma 4.4 in \cite{ArnHopHuis}, the proof essentially being the same:
  \begin{lemma}
  \label{lemma:DoubleTrace}
    For $X\in\bar\VF(\Sigma)$, define $\mathcal S^{ij}(X)=\frac{1}{\rho}\eps^{ab}(\partial_ax^i)(\bar\nabla_b X)^j$. Then for all $N,N'\in\VF^\perp(\Sigma)$ and $f,h\in C^\infty(\Sigma)$,
	\begin{equation}
	\label{eq:DoubleTrace}
	  \mathcal S^i_j(fN)\mathcal S^j_i(hN')=fh\mathcal S^i_j(N)\mathcal S^j_i(N').
	\end{equation}
	If $f=const$, then \eqref{eq:DoubleTrace} holds for arbitrary $N\in\bar\VF(\Sigma)$.
  \end{lemma}
  We concentrate on the case $M=\R^m_\nu$ with metric as above. Lemma \ref{lemma:DoubleTrace} becomes
  \[
    \bar g_i\bar g_j\{x^i,fN^j\}\{x^j,hN'^i\}=fh\bar g_i\bar g_j\{x^i,N^j\}\{x^j,N'^i\}
  \]
  for $f,h,N,N'$ as above. As a computational device, we need
  \begin{lemma}
  \label{lemma:OrthTranspose}
    Let $e_1,\ldots,e_n$ be a basis of the scalar product space $(V,g)$, $\hat g=(g_{ij})$ the matrix of $g$ in this basis and $\hat g^{-1}=(g^{ij})$ its inverse. Let $v_i=v_i^je_j$ be an orthonormal basis of $V$, $\sigma_i=g(v_i,v_i)$. Then
	\[
	  \sum_{i=1}^n v^k_ig(v_i,v_i)v^l_i=g^{kl}.
	\]
  \end{lemma}
  \begin{proof}
    $V_{ij}:=v^i_j$, $D:=\mathrm{diag}(\sigma_i)$ $\Rightarrow$ $V^T\hat g V=D$. $D^2=1$ implies $VDV^T=\hat g^{-1}$.
  \end{proof}

  \begin{theorem}
  \label{theorem:GaussMeanRnnuFull}
    The Gauss and mean curvature of a surface $\Sigma$ embedded in $\R^m_\nu$ are given by
    \begin{align}
\label{eq:GaussRnnuFull}K&=-\frac{\rho^4}{8g^2(p-1)!}\sum_L\bar g_i\bar g_r\bar g_n\eps_{jklL}\eps_{irnL}\{x^i,\{x^k,x^l\}\}\{x^j,\{x^r,x^n\}\}\\
\label{eq:MeanRnnuFull} H&=\frac{\rho^4}{8g^2(p-1)!}\sum_{L,k'}\bar g_j\bar g_k\bar g_l\eps_{irnL}\eps_{k'klL}\{x^i,x^j\}\{x^j,\{x^r,x^n\}\}\{x^k,x^l\}\partial_{k'},
	\end{align}
	where $\{\cdot,\cdot\}$ denotes the Poisson bracket \eqref{eq:PoissonBracket} on $\Sigma$.
  \end{theorem}
  \begin{proof}
	Let $E^I$, $\mu^I$, $\hat N^I$ be as in proposition \ref{prop:NormalVectors} and its proof, $\sigma^I:=\bar g(\hat N^I,\hat N^I)=(-1)^\delta\mu^I\bar g_\otimes(E^I,E^I)$. Using equation \eqref{eq:GaussRnnu} and lemma \ref{lemma:DoubleTrace},
	\[
	  K=-\frac{\rho^2}{2g}\sum_I(-1)^\delta\mu^I\bar g_\otimes(E^I,E^I)E_K^IE_L^I\bar g_i\bar g_j\{x^i,Z^{Kj}\}\{x^j,Z^{Li}\}.
	\]
	Since for $\mu^I=0$ we have $\hat N^I=0$, the factor $\mu^I$ may be omitted. Lemma \ref{lemma:OrthTranspose} implies $\sum_I E_K^I\bar g_\otimes(E^I,E^I)E_L^I=\bar g_{KL}=\delta_{KL}\bar g_K$, so we get with equation \eqref{eq:ZTensorUp}
	\begin{align*}
	  K&=(-1)^{\delta+1}\frac{\rho^2}{2g}\bar g_i\bar g_j\{x^i,Z_L^j\}\{x^j,Z^{Li}\}\\
	   &=(-1)^{\ind\bar g+1}\frac{\rho^4}{8g^2(p-1)!}\bar g_i\bar g_r\bar g_n\eps_{jklL}\eps^{irnL}\{x^i,\{x^k,x^l\}\}\{x^j,\{x^r,x^n\}\}.
	\end{align*}
	Using $\eps^{irnL}=\det(\bar g_{ij})^{-1}\eps_{irnL}$, we obtain equation \eqref{eq:GaussRnnuFull}. Equation \eqref{eq:MeanRnnuFull} is proved similarly.
  \end{proof}

\section{Acknowledgments}

The results presented in this note originate from the work on my Bachelor Thesis at the University of G\"ottingen \cite{BA}. I would like to thank my supervisor Prof.~Dorothea Bahns for many helpful discussions, comments and suggestions.



\begin{thebibliography}{1}

\bibitem{ArnHopHuis}
  Joakim Arnlind, Jens Hoppe, Gerhard Huisken, \emph{On the Classical Geometry of Embedded Surfaces in Terms of Poisson Brackets}, arXiv:1001.1604v1 [math.DG], 11.01.2010

\bibitem{BA}
  Peter Hintz, \emph{The geometry of embedded surfaces in terms of a Poisson bracket}, Bachelor Thesis, University of G\"ottingen, 2011

\bibitem{ONeill}
  Barrett O'Neill, \emph{Semi-Riemannian Geometry, with Applications to Relativity}, Academic Press, 1983

\end{thebibliography}
\end{document}